\documentclass{amsart}

\usepackage[cp1251]{inputenc}
\usepackage{amsfonts,amsmath,amsthm,amscd,amssymb,latexsym,mathrsfs,enumitem}
\usepackage[english]{babel}

\theoremstyle{plain}

\newtheorem{theorem}{Theorem}[section]
\newtheorem{lemma}[theorem]{Lemma}
\newtheorem{proposition}[theorem]{Proposition}
\newtheorem{corollary}[theorem]{Corollary}
\newtheorem{example}[theorem]{Example}
\newtheorem{definition}[theorem]{Definition}
\newtheorem{question}{Question}
\newtheorem{remark}[theorem]{Remark}

\numberwithin{equation}{section}

\begin{document}

\title{Extension of Borel maps with values in non-metrizable spaces}

\author{Olena Karlova}

\author{Volodymyr Mykhaylyuk}

\maketitle

\begin{abstract}
  We investigate the possibility of extension of $F_\sigma$-measurable and Baire-one maps from subspaces of topological spaces when these maps take values in spaces which covers by a sequence of metrizable spaces with special properties.
\end{abstract}

\section{Introduction}

Let $X$ and $Y$ be topological spaces and $f:X\to Y$ be a mapping. We say that $f$ is
\begin{itemize}
\item  {\it of the first Baire class}, $f\in {\rm B}_1(X,Y)$, if $f$ is a pointwise limit of a sequence of continuous maps between  $X$ and $Y$;

   \item {\it $F_\sigma$-measurable}, $f\in {\rm H}_1(X,Y)$, if $f^{-1}(V)$ is an $F_\sigma$-set in $X$ for any open subset $V$ of  $Y$;

   \item {\it functionally $F_\sigma$-measurable}, $f\in {\rm K}_1(X,Y)$, if $f^{-1}(V)$ is a functionally  $F_\sigma$-set in $X$ for any open subset $V$ of  $Y$;

   \item {\it barely continuous}, if for any non-empty closed set $F\subseteq X$ the restriction $f|_F$ has a point of continuity;

   \item {\it of the $\alpha$'th Baire class} for some $\alpha\in[2,\omega_1)$, $f\in{\rm B}_\alpha(X,Y)$, if $f$ is a pointwise limit of a sequence of maps from lower Baire classes;
\end{itemize}

We use the term ''barely continuous'' due to Stephens~\cite{steph}. Barely continuous functions are also mentioned in literature as functions with the ''point of continuity property'' \cite{Kum, Spurny}.

Let $P$ be a property of maps and let  $P(X)$ be the collection of all real-valued functions with the property $P$ defined on a topological spaces   $X$. A subspace  $E$ of a topological space  $X$ is called  {\it $P$-embedded} in $X$, if any function  $f\in  P(E)$ can be extended to a function   $g\in P(X)$.

 The extension problem if a property $P$ is weaker than continuity (i.e., if $P={\rm B}_\alpha$ or $P={\rm H}_\alpha$) was solving by many mathematicians and were started in classical works of Hausdorff~\cite{Hausdorff:1957}, Sierpi\'{n}ski~\cite{Sierp:ext}, Alexits~\cite{Alex}, Hahn~\cite{Hahn:ext}, Kuratowski~\cite{Ku2}. Further, these investigation were continued in works of Hansell~\cite{Hansell:1974}, Spurn\'{y}~\cite{Spurny:ext1, Spurny:ext}, Kalenda and Spurn\'{y}~\cite{KaSp}, Karlova~\cite{Karlova:BAMS, Karlova:CMUC:2013}. Let us observe that all of these authors deal with maps taking values in metrizable spaces, while the extension problem for functions with values in non-metrizable spaces was not studied at all. Therefore, we start the investigations in this direction with this paper. 
 
 In Section~\ref{sec:11} we introduce a class of topological spaces which admit sequentially absorbing coverings. This class includes strict inductive limits of sequences of locally convex metrizable spaces and small box products of sequences of metrizable spaces. We study separation axioms and metrizability of these spaces. Section~\ref{sec:22} contains technical auxiliary results about $F_\sigma$-measurable and Baire-one maps with values in  spaces which admit sequentially absorbing coverings. Further, we introduce $\sigma$-regular maps in Section~\ref{sec:33} and investigate their relations with $F_\sigma$-measurable, Baire-one and barely continuous maps. Finally, Section~\ref{sec:44} contains Extension Theorem for Baire-one and $F_\sigma$-measurable maps with values in strongly $\sigma$-metrizable spaces with additional properties.  We finish the paper with an open question.

\section{Some properties of spaces which admit sequentially absorbing coverings}\label{sec:11}

\begin{definition}
{\rm A countable covering $(X_n:n\in\omega)$ of a topological space $X$ is said to be {\it sequentially absorbing} if for any convergent sequence $(x_n)_{n\in\omega}$ in $X$ there exists $k\in\omega$ such that $\{x_n:n\in\omega\}\subseteq X_k$.}
\end{definition}
 If we put $Y_n=\overline{X_0\cup\dots\cup X_n}$ for every $n\in\omega$, then we obtain an increasing covering $(Y_n:n\in\omega)$ of $X$ which is, evidently, sequentially absorbing. Therefore, from now on we assume that every sequentially absorbing covering is closed and increasing.

\begin{definition}
  {\rm A topological space $X$ is said to be {\it (strongly) $\sigma$-metrizable} if it admits a (sequentially  absorbing) covering by metrizable subspaces. }
\end{definition}

It is worth noting that the class of all strongly $\sigma$-metrizable spaces includes countable $T_1$-spaces in which all convergent sequences are trivial,  strict inductive limits of sequences of locally convex metrizable spaces and small box products of sequences of metrizable spaces.

A topological space $X$ is said to be {\it perfect} if every its closed subset is $G_\delta$.

\begin{proposition}\label{pro:properties_sm} Let $(X_n:n\in\omega)$ be a sequentially absorbing covering of a topological space $X$.
\begin{enumerate}
  \item\label{it:pro:properties_sm:1} If $Y$ is a topological space and $f:X\to Y$ be a map such that $f|_{X_n}$ is continuous for all $n\in\omega$, then $f$ is sequentially continuous on $X$.

  \item\label{it:pro:properties_sm:6} If $Y$ is a first countable topological space and $f:Y\to X$ is continuous, then for any $y\in Y$ there exist a neighborhood $U$ of $y$ in $Y$ and a number $n\in\omega$ such that $f(U)\subseteq X_n$.

\item\label{it:pro:properties_sm:3} If each $X_n$ is Hausdorff and $X$ is a first countable space, then $X$ is Hausdorff.

\item\label{it:pro:properties_sm:4} If each $X_n$ is regular and $X$ is a Fr\'{e}chet-Urysohn  space, then $X$ is regular.

  \item\label{it:pro:properties_sm:2} If each $X_n$ is  (perfectly) normal and $X$ is a sequential space, then $X$ is (perfectly) normal.

 \item\label{it:pro:properties_sm:5} If each $X_n$ is metrizable and $X$ is a first countable space, then $X$ is metrizable.
\end{enumerate}
\end{proposition}

\begin{proof}
{\bf \ref{it:pro:properties_sm:1}).} It implies from definitions immediately.

{\bf \ref{it:pro:properties_sm:6}).} We fix a point $y\in Y$ and a base $(U_n)_{n\in\omega}$ of neighborhoods of $y$. Assume the contrary and choose a sequence $(y_{n})_{n\in\omega}$ of points from $Y$ such that $y_{n}\in U_n$ and $f(y_{n})\not\in X_n$ for all $n\in\omega$. Since $f$ is continuous at $y$ and  $y_n\to y$ in $Y$, $f(y_n)\to f(y)$ in $X$. Then there exists $m\in\omega$ such that $\{f(y_n):n\in\omega\}\subseteq X_m$, a contradiction.

{\bf \ref{it:pro:properties_sm:3}).} Let $(x_n)_{n\in\omega}$ be a convergent sequence in $X$. Then there exists $N\in\omega$ such that $\{x_n:n\in\omega\}\subseteq X_N$. Since $X_N$ is Hausdorff, the sequence $(x_n)_{n\in\omega}$ has a unique limit point. Therefore, $X$ is Hausdorff by~\cite[1.6.16]{Eng-eng}.

{\bf \ref{it:pro:properties_sm:4}).} Fix a point $a\in X$ and a closed nonempty set $F\subseteq X$ such that $a\not\in F$. Without loss of generality we may assume that $a\in X_0$. Since $X_0$ is regular, there exists an open in $X_0$ set $U_0\subseteq X_0$ such that
\begin{gather*}
 F_0= F\cap X_0\subseteq U_0\quad\mbox{and}\quad a\not\in \overline{U_0}.
\end{gather*}

Now we show that there exists an open in $X_1$ set $U_1\subseteq X_0$ such that
\begin{gather*}
U_1\cap X_0=U_0,\quad F_1= F\cap X_1\subseteq U_1\quad\mbox{and}\quad a\not\in \overline{U_1}.
\end{gather*}
Take an open in $X_1$ set $V_1\subseteq X_1$ such that $V_1\cap X_0=U_0$. Using the regularity of $X_1$ we choose an open in $X_1$ set $W_1\subseteq X_1$ such that $F_1\cup\overline{U_0}\subseteq W_1$ and $a\not\in \overline{W_1}$. Now we put
\begin{gather*}
U_1=W_1\cap(V_1\cup (X_1\setminus X_0)).
\end{gather*}
It is obvious that $a\not\in \overline{U_1}$. Since $X_0$ is closed, $U_1$ is open in $X_1$. Moreover,
\begin{gather*}
U_0\subseteq W_1\cap V_1 \subseteq U_1\quad\mbox{and}\quad
U_1\cap X_0\subseteq V_1\cap X_0=U_0.
\end{gather*}
Therefore, $U_1\cap X_0=U_0$. Since
\begin{gather*}
F_1\subseteq W_1\quad\mbox{and}\quad
F_1=F_0\cup(F_1\setminus F_0)\subseteq U_0\cup (X_1\setminus X_0)\subseteq V_1\cup (X_1\setminus X_0),
\end{gather*}
$F_1\subseteq U_1$.

Repeating this process we obtain a sequence $(U_n)_{n\in\omega}$ such that
\begin{gather}
  U_n\subseteq X_n\,\,\mbox{is open in}\,\,X_n,\label{gath:propU1}\\
  U_{n+1}\cap X_n=U_n,\label{gath:propU2}\\
  a\not\in\overline{U_n},\label{gath:propU3}\\
  F_n=F\cap X_n\subseteq U_n\label{gath:propU4}
\end{gather}
for all $n\in\omega$.

We put
$$
U=\bigcup_{n\in\omega} U_n.
$$
Assume that $a\in\overline{U}$. Then there exists a sequence $(x_n)_{n\in\omega}$ such that $x_n\in U$ for all $n\in\omega$ and $x_n\to a$, since $X$ is a Fr\'{e}chet-Urysohn  space. There exists a number $k\in\omega$ such that $\{x_n:n\in\omega\}\cup\{a\}\subseteq X_k$. Properties~(\ref{gath:propU1}) and (\ref{gath:propU2}) imply that $U\cap X_k=U_k$. Then $x_n\in U_k$ for all $n\in\omega$. Hence, $a\in\overline{U_k}$, which contradicts to (\ref{gath:propU3}).

Property (\ref{gath:propU4}) implies that $F\subseteq U$.

{\bf \ref{it:pro:properties_sm:2}).} We will prove this property for the case of perfectly normal spaces.

We fix a closed set $F$ in $X$ and show that there exists a continuous function $f:X\to [0,+\infty)$ such that $F=f^{-1}(0)$.

Since $X_0$ is perfectly normal, there exists a continuous function $f_0:X_0\to [0,1]$ such that $F\cap F_0=f_0^{-1}(0)$. Using Tietze Extension Theorem, we can extend a continuous function $g_0:X_0\cup(F\cap X_1)\to [0,1]$,
\begin{gather*}
  g_0(x)=\left\{\begin{array}{ll}
                  0, & x\in (F\cap X_1)\setminus X_0, \\
                  f_0(x), & x\in X_1,
                \end{array}
  \right.
\end{gather*}
to a continuous function $\varphi_1:X_1\to [0,1]$. Moreover, there exists a continuous function $\psi_1:X_1\to [0,1]$ such that $X_0\cup(F\cap X_1)=\psi_1^{-1}(0)$. Now for every $x\in X_1$ we put
\begin{gather*}
  f_1(x)=\varphi_1(x)+\psi_1(x).
\end{gather*}
Then $f_1:X_1\to [0,2]$  is a continuous extension of $f_0$ and $F\cap X_1=f_1^{-1}(0)$. Repeating this process inductively, we obtain a sequence $(f_n:n\in\omega)$ of continuous functions $f_n:X_n\to [0,n+1]$ such that
\begin{gather}
\label{gath:1}f_{n+1}|_{X_n}=f_n,\\
F\cap X_n=f_n^{-1}(0)
\end{gather}
for all $n\in\omega$. Now for every $x\in X$ let
\begin{gather*}
  f(x)=f_n(x)\quad\mbox{if}\quad x\in X_n.
\end{gather*}
Notice that the function $f:X\to [0,+\infty)$ is well defined because of (\ref{gath:1}). It is easy to see that $F=f^{-1}(0)$.

Property~(\ref{it:pro:properties_sm:1}) of the Proposition implies that $f$ is sequentially continuous on $X$. Consequently, $f$ is continuous on $X$.
Let us observe that $X$ is a $T_1$-space, since each $X_n$ is a $T_1$-space. Therefore, by the Vedenissoff's Theorem~\cite[1.5.19]{Eng-eng}, $X$ is perfectly normal.

The case of normal spaces $X_n$ can be proved similarly.

{\bf \ref{it:pro:properties_sm:5}).} Let $d_0$ be a metric which generates the topology of $X_0$. Since $X_0$ is closed subspace of a metrizable  space $X_1$,
 Hausdorff's theorem on extending metrics~\cite{Haus} implies that there exists a metric $d_1$ on $X_1$ such that $d_1$ is an extension of $d_0$ and generates the topology of $X_1$. Proceeding in this way we obtain a sequence $(d_n)_{n\in\omega}$ of metrics such that $d_{n+1}$ is an extension of $d_n$ and $d_n$ generates the topology of $X_n$ for all $n\in\omega$. We put $n(x,y)=\min\{k\in\omega:x,y\in X_k\}$ and define a metric $d$ on $X$ by the rule $d(x,y)=d_{n(x,y)}(x,y)$ for all $x,y\in X$.

We show that $d$ generates the topology of $X$. Fix $a\in X$ and let $(U_k)_{k\in\omega}$ be a base of neighborhoods of $a$. We claim that there exists $n\in\omega$ such that $a\in {\rm int}X_n$. Assume the contrary. Then we can choose an increasing sequence of numbers $(k_n)_{n\in\omega}$ and a sequence of points  $(x_n)_{n\in\omega}$ such that $x_n\in U_{k_n}\setminus X_n$ for every $n\in\omega$. Clearly, $x_n\to a$. It follows that there is $m\in\omega$ such that $\{x_n:n\in\omega\}\subseteq X_m$, a contradiction. So, $a\in {\rm int}X_n$ for some $n\in\omega$.  Since $d|_{X_n}$ generates the topology of $X_n$, the metric topology on $X$ and the initial one coincides.
\end{proof}

It is well-known  (see  \cite[p.~386]{Kuratowski:Top:1}) that any Baire-one map $f:X\to Y$ between a topological space $X$ and a perfectly normal space $Y$ is functionally $F_\sigma$-measurable. Therefore, we have the following application of Proposition~\ref{pro:properties_sm} to the classification of Baire-one maps.

\begin{corollary}\label{cor:2}
Let   $X$ be a topological space, $(Y_n:n\in\omega)$ be a sequentially absorbing covering of a sequential space $Y$ by  perfectly normal subspaces. Then  ${\rm B}_1(X,Y)\subseteq {\rm K}_1(X,Y)$.
\end{corollary}

The example below shows that the condition on a covering $(Y_n:n\in\omega)$ to be sequentially absorbing is essential in Corollary~\ref{cor:2}.
\begin{example}\label{ex:Nem}
  Let $\mathbb P$  be the Niemytski plane, i.e. $\mathbb P=\mathbb R\times [0,+\infty)$, where a base of neighborhoods of $(x,y)\in \mathbb P$ with $y>0$ form open balls with the center in $(x,y)$, and a base of neighborhoods of $(x,0)$ form the sets $U\cup\{(x,0)\}$, where $U$ is an open ball which tangent to $\mathbb R\times \{0\}$ in the point $(x,0)$.  Then ${\rm B}_1(\mathbb R,\mathbb P)\not\subseteq {\rm H}_1(\mathbb R,\mathbb P)$.

  In particular, $\mathbb P$ is not perfectly normal.
\end{example}

\begin{proof}
Define the function $f:{\mathbb R}\to {\mathbb P}$ by $f(x)=(x,0)$ for all $x\in\mathbb R$.

We put $f_n(x)=(x,\frac{1}{n})$ for all  $x\in {\mathbb R}$ and $n\in
{\mathbb N}$. It is easy to see that the functions $f_n:{\mathbb R}\to {\mathbb P}$ are continuous and $f_n(x)\to f(x)$ on ${\mathbb R}$. Hence, $f\in
{\rm B}_1({\mathbb R},{\mathbb P})$.

Consider the set $F=\{(r,0):r\in {\mathbb Q}\}$, where ${\mathbb
Q}$ is the set of all rational numbers. The set ${\mathbb P}_0=\{(x,0):x\in {\mathbb R}\}$ is a discrete closed subspace of
${\mathbb P}$, and so $F$ is closed in ${\mathbb P}$. But $f^{-1}(F)={\mathbb Q}$ is not a $G_\delta$-set in ${\mathbb R}$.
Thus, $f\not\in {\rm H}_1(X,Y)$.

Notice that $\mathbb P$ admits a covering $(\mathbb P_0\cup (\mathbb R\times [\tfrac 1n,+\infty)):n\in\mathbb N)$ by closed metrizable subspaces.
\end{proof}

\begin{remark} {\rm The function $f$ from Example~\ref{ex:Nem} is not Borel measurable. In particular, for Borel nonmeasurable set $A\subseteq \mathbb R$ we have $ A=f^{-1}(B)$ while the set $\{(x,0):x\in A\}$ is closed in $\mathbb P$.}
\end{remark}

The following example shows that conditions like the first countability axiom of the range space in the previous result are essential.

\begin{example}\label{ex:firstcount} There exist a countable space $X$ such that
\begin{enumerate}
  \item[1)] $X$ is a $T_1$-space;

  \item[2)] any convergent sequence in $X$ is trivial;

  \item[3)] $X$ is a strongly $\sigma$-metrizable space;

  \item[4)] $X$ is not Hausdorff.
\end{enumerate}
\end{example}

\begin{proof} Let $B=[0,1]\cap \mathbb Q$ and $\mathscr A$ be the system of all nowhere dense $G_\delta$-subsets of $B$. It is clear that
\begin{enumerate}
  \item[(a)] $A_1\cup A_2\in\mathscr A$ for every $A_1,A_2\in\mathscr A$;

  \item[(b)] if $A_1\in\mathscr A$ and $A_2\subseteq A_1$ then $A_2\in\mathscr A$;

  \item[(c)] for every infinite set $C\subseteq B$ there exists an infinite set $A\subseteq C$ such that $A\in\mathscr A$.
  \end{enumerate}

Let $\varphi:B\to\omega$ be a bijection and
$$X=\omega \cup \{x_*, x^*\},$$
where $x_*\ne x^*$ and $x_*, x^*\not\in \omega$.

Let us describe a topology on $X$. All points $x\in\omega$ are isolated in $X$. Moreover, a set $U\subseteq X$ is a basic neighborhood of $x_*$ ($x^*$) in $X$ if
$$
U=\{x_*\}\cup \varphi(B\setminus A)
$$
$$
(U=\{x^*\}\cup \varphi(B\setminus A))
$$
for some $A\in\mathscr A$.

It is obvious that $X$ is a $T_1$-space.

Now we verify $2)$. Let $(z_n)_{n\in\omega}$ be a convergent  sequence in $X$ and $z=\lim\limits_{n\to\infty} z_n$. If $z\in\omega$ then there exists $n_0\in\omega$ such that $z_n=z$ for every $n\geq n_0$. It remains to consider the case $z\in\{x_*,x^*\}$. Suppose that $z=x_*$ and the set
$$
N=\{n\in\omega:z_n\ne x_*\}
$$
is infinite. Therefore, the set
$$
M=\{z_n:n\in N\}
$$
is infinite too. According to $(c)$ there exists an infinite set $A\in\mathscr A$ such that $$A\subseteq \varphi^{-1}(M\setminus\{x^*\}).$$ Then the set
$$
N_1=\{n\in N:\varphi^{-1}(z_n)\in A\}
$$
is infinite and
$$z_n\not\in U=\varphi(B\setminus A)\cup \{x_*\}$$
for every $n\in N_1$. But it contradicts to the equality $\lim\limits_{n\to\infty} z_n=x_*$.

In the case of $z=y^*$ we   argue analogously.

It follows from $2)$ that the sequence $(X_n)_{n\in\omega}$ of discrete subspaces
  $$
  X_n=\{x_*, x^*\}\cup \{k:0\leq k \leq n\}
  $$
of $X$ is a sequentially absorbing covering of $X$. Thus $X$ is a strongly $\sigma$-metrizable space.

It remains to note that
$$
U_*\cap U^*\ne \emptyset
$$
for every neighborhoods $U_*$ and $U^*$ of $x_*$ and $x^*$ respectively.
\end{proof}

\section{Auxiliary facts}\label{sec:22}

We define a function $\cdot^+:[0,\omega_1)\to [0,\omega_1)$,
$$
\alpha^+=\left\{\begin{array}{ll}
                  \alpha, & \alpha\in [0,\omega), \\
                  \alpha+1, &  \alpha\in [\omega,\omega_1).
                \end{array}
\right.
$$

\begin{proposition}\label{prop:ext:Balpha} Let $\alpha\in [1,\omega_1)$,  $X$ be a topological space, $(Y_n:n\in\omega)$ be a sequentially absorbing covering of a topological space $Y$ and  $f\in{\rm B}_{\alpha}(X,Y)$. Then there exists a partition $(X_n:n\in\omega)$ of $X$ by functionally ambiguous sets $X_n$ of the $\alpha^+$'th class such that $f(X_n)\subseteq Y_n$  and $f|_{X_n}\in {\rm B}_\alpha(X_n,Y_n)$ for every $n\in\omega$.
\end{proposition}

\begin{proof} We choose a sequence  $(f_n)_{n\in\omega}$ of maps $f_n\in{\rm B}_{\alpha_n}(X,Y)$ which converges to $f$ pointwisely on $X$ and $\alpha_n<\alpha$ for all $n\in\omega$. We put
$$
C_n=\bigcap_{k\ge n}f_k^{-1}(Y_k)
$$
 for $n\in\omega$.  If $\alpha<\omega_0$, then $f_k\in{\rm B}_{\alpha-1}(X,Y)\subseteq {\rm K}_{\alpha-1}(X,Y)$ and, consequently,  $C_n\in\mathcal M_{\alpha-1}(X)$; if  $\alpha=\gamma+1\ge \omega_0$ is isolated, then $f_k\in{\rm B}_\gamma(X,Y)\subseteq {\rm K}_{\alpha}(X,Y)$ and $C_n\in\mathcal M_{\alpha}(X)$; if $\alpha$ is limit, then  $\alpha_k+1<\alpha$, $f_k\in{\rm B}_{\alpha_k}(X,Y)\subseteq {\rm K}_{\alpha_k+1}(X,Y)$ and $C_n\in\mathcal M_{\alpha}(X)$. In any case the set $C_n$ is functionally ambiguous of the class $\alpha^+$ in $X$.

Notice that the family $(C_n:n\in\omega)$ is a covering of $X$. Indeed, fix $x\in X$. Since the sequence $(f_k(x))_{k=1}^\infty$ is convergent in $Y$, there exists a number $n_0$ such that $\{f_k(x):k\in\omega\}\subseteq Y_n$ for all $n\ge n_0$. Hence, $x\in C_{n_0}$.

We put $X_0=C_0$ and $X_n=C_n\setminus \bigcup_{k<n} C_k$ for all $n\ge 1$. Then $(X_n:n\in\omega)$ is a partition of $X$ by functionally ambiguous sets of the class $\alpha^+$.

Moreover, if $x\in C_n$, then   $f_k(x)\in Y_n$ for all $k\ge n$. Since $\lim\limits_{k\to\infty} f_k(x)=f(x)$ and $Y_n$ is closed, $f(x)\in Y_n$.
\end{proof}

\begin{definition}
  {\rm A covering $(X_n:n\in\omega)$ of a topological space $X$ is called {\it super sequentially absorbing} if for every sequence $(x_n)_{n\in\omega}$  such that $x_n\not\in X_n$ for all $n\in\omega$ the set $\{x_n:n\in\omega\}$ is closed in $X$.}
\end{definition}

It is easy to see that every super  sequentially absorbing covering is sequentially absorbing.

Let us observe that strict inductive limits of sequences of locally convex metrizable spaces, or  small box products of sequences of metrizable spaces admits super  sequentially absorbing coverings. Moreover, any sequentially absorbing covering of a first countable space is super sequentially absorbing.

\begin{proposition}\label{prop:aux_sigma_metr}
Let $X$ be a topological space, \mbox{$\mathscr Y=(Y_n:n\in\omega)$} be a sequentially absorbing covering of a topological space $Y$ and $f:X\to Y$ be a map. If
\begin{enumerate}
  \item[1)] $X$ is a hereditarily Baire  metrizable separable space, $\mathscr Y$ is super sequentially absorbing and $f\in{\rm H}_1(X,Y)$, or

  \item[2)] $X$ is a first countable perfect paracompact space and $f$ is barely continuous,
\end{enumerate}
then there exists a covering $(C_n:n\in\omega)$ of $X$ by $F_\sigma$-sets $C_n$ such that $f(C_n)\subseteq Y_n$ for all $n\in\omega$.
\end{proposition}

\begin{proof} {\bf 1).} Let  $\mathscr B=(B_k:k\in\omega)$ be a countable base of the space $X$.

We show that there is $n\in\omega$ such that ${\rm int}\, f^{-1}(Y_n)\ne\emptyset$. Assume the contrary. Then we claim that every set $A_n=f^{-1}(Y\setminus Y_n)$ contains a dense $G_\delta$-subset of $X$. Indeed, fix $n\in\omega$ and put $n_0=n$. According to our assumption the set $A_{n_0}$ is   dense in $X$. Therefore, there exist a point $x_1\in B_0\cap A_{n_0}$ and a number $n_1>n_0$ such that $f(x_1)\in Y_{n_1}\setminus Y_{n_0}$. Then, since $A_{n_1}$ is dense in $X$, there exist a point $x_2\in B_1\cap A_{n_1}$ and a number $n_2>n_1$ such that $f(x_2)\in Y_{n_2}\setminus Y_{n_1}$. Repeating this process, we obtain an increasing sequence of numbers
$$
n=n_0<n_1<\dots<n_k<\dots
$$
and a sequence $(x_k)_{k\in\mathbb N}$ of points in $X$ such that
$$
x_k\in B_{k-1}\quad\mbox{and}\quad f(x_k)\in Y_{n_k}\setminus Y_{n_{k-1}}
$$
for all $k\in\mathbb N$. Since the covering $\mathscr Y$ is super sequentially absorbing, the set
$$
F=\{f(x_k):k\in\mathbb N\}
$$
is closed in  $Y$. Then $f^{-1}(F)$ is a  $G_\delta$-set in $X$, since $f\in{\rm H}_1(X,Y)$. Moreover,
$$
\{x_k:k\in\mathbb N\}\subseteq f^{-1}(F)\subseteq A_{n}.
$$
Hence, $A_{n}$ contains everywhere dense $G_\delta$-subset of $X$.

Since $X$ is a  Baire space, $\bigcap_{n\in\omega} A_n\ne\emptyset$. On the other side,
$$
\bigcap_{n\in\omega} A_n=\bigcap_{n\in\omega} f^{-1}(Y\setminus Y_n)=\emptyset,
$$
which implies a contradiction.

Therefore, there exists $k_0\in {\omega}$ such that  ${\rm int}\, f^{-1}(Y_n)\ne\emptyset$ for all $n\ge k_0$.

We put $E_{n,0}={\rm int}\,f^{-1}(Y_n)$ for all $n\in\omega$, $E_0=\bigcup_{n\in\omega}^\infty G_{n,0}$ and $F_0=X$. Since $X$ is hereditarily Baire, $F_1=F_0\setminus E_0$ is a  closed  Baire subspace of $F_0$. Consider a map $f_1=f|_{F_1}\in {\rm H}_1(F_1,Y)$ and put
$E_{n,1}={\rm int}_{F_1}f_1^{-1}(Y_n)$ for every  $n\in\omega$. By similar arguments one can show that the set   $E_1=\bigcup_{n=1}^\infty E_{n,1}$ is nonempty and open in $F_1$. Since $F_1$ is closed in $X$,  $E_1$ is an  $F_\sigma$-set in $X$.

Assume that for some $\alpha<\omega_1$ the sets $F_\xi$ and $E_\xi$ are constructed for all $\xi<\alpha$ and
\begin{enumerate}
\item[(i)] $F_\xi$ is closed in  $X$ for all  $\xi<\alpha$;

\item[(ii)] $F_\xi\subseteq F_\eta$ for all for all  $\eta<\xi<\alpha$;

\item[(iii)] $E_\xi=\bigcup_{n=1}^\infty E_{n,\xi}$, where $E_{n,\xi}={\rm int}_{F_\xi}f_\xi^{-1}(Y_n)$ and $f_\xi=f|_{F_\xi}$;

\item[(iv)] if $F_\xi\ne\emptyset$, then  $E_\xi\ne\emptyset$;

\item[(v)] if $\xi=\eta+1$, then   $F_\xi=F_\eta\setminus E_\eta$; if $\xi$ is limit, then $F_\xi=\bigcap_{\eta<\xi}F_\eta$.
\end{enumerate}
We put
\begin{gather*}
F_\alpha=\left\{\begin{array}{ll}
                  F_\xi\setminus E_\xi, & \alpha=\xi+1, \\
                  \bigcap\limits_{\xi<\alpha}F_\xi, & \alpha\,\,\, \mbox{is limit}.
                \end{array}
\right.
\end{gather*}
Let $f_\alpha=f|_{F_\alpha}$, $E_{n,\alpha}={\rm int}_{F_\alpha}f_\alpha^{-1}(Y_n)$ and $E_\alpha=\bigcup_{n=1}^\infty E_{n,\alpha}$.

Since $(F_\xi:\xi<\omega_1)$ is a strictly decreasing sequence of closed sets in the second countable space $X$, there exists  $\alpha<\omega_1$ such that $F_\alpha=\emptyset$. We denote $\beta=\inf\{\alpha<\omega_1:F_{\alpha}=\emptyset\}$. Then
$$
X=\bigcup_{\xi<\beta}(F_\xi\setminus F_{\xi+1})=\bigcup_{\xi<\beta}E_\xi=\bigcup_{n\in\omega} \bigcup_{\xi<\beta}E_{n,\xi}.
$$
We put
$$
C_n=\bigcup_{\xi<\beta} E_{n,\xi}.
$$
Then $C_n$ is an $F_\sigma$-subset of $X$. Moreover, $X=\bigcup_{n\in\omega} C_n$ and $f(C_n)\subseteq Y_n$ for all $n\in {\omega}$.

{\bf 2).} Let $\mathscr G$ be a collection of all open subsets of $X$ satisfying the following condition: for every $G\in\mathscr G$ there exists a countable covering $\mathscr C$ of $G$ by $F_\sigma$-sets such that
\begin{gather}\label{gath:prec}
\mathscr C\prec (f^{-1}(Y_n):n\in\omega).
\end{gather}
Let $U=\bigcup_{G\in\mathscr G}G$.

We prove that
\begin{gather}\label{gath:imp}
\forall\,\, O\subseteq X, O\,\,\mbox{is open in}\,\,X, \,\, O\subseteq U \Longrightarrow O\in\mathscr G.
\end{gather}
Fix an open set $O\subseteq U$ and notice that $O$ is a paracompact space as an $F_\sigma$-subspace of $X$.  Then there  exists a locally finite   covering $\mathscr V$ of $O$ by open subsets such that $\mathscr V\prec \mathscr G$. For every $V\in\mathscr V$ we take $G_V\in\mathscr G$ with $V\subseteq G_V$. Let $\mathscr C_{V}=(C_{V,n}:n\in\omega)$ be a countable covering of $G_V$ by $F_\sigma$-sets satisfying~(\ref{gath:prec}). We put $B_{V,n}=V\cap C_{V,n}$ for all $V\in\mathscr V$ and $n\in\omega$. Then every $B_{V,n}$ is an $F_\sigma$-subset of $O$. It is easy to see that the family  $(B_{V,n}:V\in\mathscr V)$ is locally finite in $O$. Now let $B_n=\bigcup_{V\in\mathscr V} B_{V,n}$ for all $n\in\omega$. Then every $B_n$ is an $F_\sigma$-subset of $O$ (and of $X$) as a union of a locally finite family of $F_\sigma$-sets. Moreover, $(B_n:n\in\omega)\prec (f^{-1}(Y_n):n\in\omega)$ and $\bigcup_{n\in\omega} B_n=O$. Hence, $O\in\mathscr G$.

Now we show that $U=X$. To obtain a contradiction, assume that the set $F=X\setminus U$ is not empty. Take a point $x_0\in F$ of continuity of the restriction $f|_F$. Then Proposition~\ref{pro:properties_sm}~(\ref{it:pro:properties_sm:6}) implies that there exist an open neighborhood $O$ in $X$ and a number $m\in\omega$ such that $f(O\cap F)\subseteq Y_m$. Since $O\setminus F$ is an open subset of $U$, property (\ref{gath:imp}) implies that $O\setminus F\in\mathscr G$. Then $O=(O\cap F)\cup(O\setminus F)\in\mathscr G$. But $O\setminus U\ne\emptyset$, a contradiction.

Finally, property (\ref{gath:imp}) implies that $U\in\mathscr G$.
\end{proof}

\begin{remark}
{\rm  It is clear that Proposition~\ref{prop:aux_sigma_metr} remains valid if $X$ admits a countable covering by $F_\sigma$-subsets $X_n$ such that every $X_n$ is hereditarily Baire metrizable separable space or first countable perfect paracompact space.}
\end{remark}

The following   example shows that an analog of Proposition \ref{prop:aux_sigma_metr} is not true even for countable $Y$ and for $f\in {\rm B}_1(\mathbb R,Y)$.

 \begin{example} For the Riemann function $f\in {\rm B}_1(\mathbb R\,[0,1])$,
 $$ f(x)=\left\{\begin{array}{ll}
                         0, & x\not\in \mathbb Q;\\
                         \frac1n, & x=\frac mn\,\,\,\, {\rm is}\,\,\,\,{\rm irreducible},
                       \end{array}
 \right.
$$
there is no covering $(C_n:n\in\omega)$ of $[0,1]$ by $F_\sigma$-sets  such that
$$
f(C_n)\subseteq Y_n=\{0\}\cup [\tfrac{1}{n+1},1]
$$
for all $n\in\omega$.
\end{example}

Next example shows that Proposition \ref{prop:aux_sigma_metr} is not true for strongly $\sigma$-metrizable~$Y$.

 \begin{example} There exists a countable space $Y$ such that
\begin{enumerate}
  \item[1)] $Y$ is a perfectly normal space;

  \item[2)] any convergent sequence in $Y$ is trivial;

  \item[3)] $Y$ is a strongly $\sigma$-metrizable space;

  \item[4)] there exists a map $f\in {\rm H}_1([0,1], Y)$ such that for every sequentially absorbing covering $(Y_n:n\in\omega)$ of $Y$ by metrizable subspaces $Y_n$ there is no covering $(C_n:n\in\omega)$ of $[0,1]$ by $F_\sigma$-sets  with $f(C_n)\subseteq Y_n$ for all $n\in\omega$.
\end{enumerate}

\end{example}

\begin{proof} Let
$$
Y=\{y_*\}\cup \omega
$$
be the space constructed in Example~\ref{ex:firstcount}. Clearly, $Y$ satisfies  conditions 1) -- 3).

We proof that property $4)$ holds. Consider the function $f:[0,1]\to Y$,
$$ f(x)=\left\{\begin{array}{ll}
                         \varphi(x), & x\in B;\\
                         y_*, & x\in [0,1]\setminus B.
                       \end{array}
 \right.
$$

Let us observe that the set
$$
D=f^{-1}(Y\setminus Z)
$$
is dense in $[0,1]$ for every metrizable subspace $Z$ of $Y$. Indeed, if $y_*\not\in Z$, then
$[0,1]\setminus B\subseteq D$ and $D$ is dense in $[0,1]$. Now let $y_*\in Z$. It follows from $2)$ that $Z$ is a discrete space. Therefore, there exists a neighborhood $U_0$ of $y_0$ in $Y$ such that
$$
U_*\cap Z=\{y_*\},
$$
that is, there exists $A_0\in\mathscr A$ such that
$$
\varphi(B\setminus A_0)\cap Z=\emptyset.
$$
Therefore,
$$
B\setminus A_0\subseteq f^{-1}(Y\setminus Z)=D.
$$
Since $A_0$ is nowhere dense in $[0,1]$ and $B$ is dense in $[0,1]$, $D$ is dense in $[0,1]$.

Assume that there exist a sequentially absorbing covering $(Y_n)_{n\in\omega}$ of $Y$ by metrizable subspaces and a covering $(C_n:n\in\omega)$ of $[0,1]$ by $F_\sigma$-sets with $f(C_n)\subseteq Y_n$ for every $n\in\omega$. Then each  $D_n=X\setminus C_n$ is a dense $G_\delta$-subset of $[0,1]$. Therefore, the set
$\bigcap_{n\in\omega} D_n$ is dense in $[0,1]$. It follows that $\bigcup_{n\in\omega} C_n\ne [0,1]$, a contradiction.
\end{proof}

\section{Relations between barely continuous,  $F_\sigma$-measurable and Baire-one maps}\label{sec:33}
\begin{definition}
  {\rm A set $A\subseteq X$ is called {\it resolvable in the sense of Hausdorff} or {\it an $H$-set}~\cite[12.V]{Kuratowski:Top:1}, if for every closed set $F\subseteq X$ the set $\overline{F\cap A}\cap\overline{F\setminus A}$ is nowhere dense in $F$.
}\end{definition}

\begin{definition}
{  A map $f:X\to Y$ between topological spaces $X$ and  $Y$ is called  {\it $H_\sigma$-measurable}, if the preimage of any open  set   $V$ in $Y$ is a union of a sequence of $H$-subsets of $X$.}
\end{definition}

Since any open or closed set is an $H$-set, any $F_\sigma$-set is $H_\sigma$. Hence, every $F_\sigma$-measurable map between topological spaces is $H_\sigma$-measurable.

If $X$ is a perfect paracompact space, then any $H$-set in $X$ is $F_\sigma$ and $G_\delta$ simultaneously; the converse is true in a hereditarily Baire space $X$~\cite{Karlova:FRes:2011}.

Koumoullis~\cite{Kum} proved that every barely continuous map $f:X\to Y$ is $H_\sigma$-measurable if $X$ is a topological space and $Y$ is a metric space.
Moreover, it was shown in~\cite[Theorem 4.12]{Kum} that for a map $f:X\to Y$ between a hereditarily Baire metric space $X$ and a metric space $Y$ the following conditions are equivalent:
\begin{itemize}
  \item $f$ is barely continuous;

  \item $f$ is $F_\sigma$-measurable.
\end{itemize}

We develop the results of Koumoullis in this section.

The  fact below is proved completely similarly to Theorem~2~\cite[p.~395]{Kuratowski:Top:1}.
\begin{proposition}\label{prop:barely_is_Hmeas}
 Every barely continuous map $f:X\to Y$ between a topological space $X$ and a perfect space $Y$ is  $H_\sigma$-measurable.
\end{proposition}

\begin{proof}
         Let $F$ be a closed subset of $Y$. Since $Y$ is perfect, there exists a sequence $(F_n)_{n\in\omega}$ of closed subsets of $Y$ such that    $G=Y\setminus F=\bigcup_{n\in\omega} F_n$.  We denote $E=f^{-1}(F)$ and $H_n=f^{-1}(F_n)$ for  $n\in\omega$.

We fix $n\in\omega$ and show that there exists an $H$-set  $B\subseteq X$ such that
   $E\subseteq B\subseteq X\setminus H_n$. Assume that it is not the case. Then (see \cite[12.III]{Kuratowski:Top:1}) there exists a closed non-empty set  $A\subseteq X$ such that
   \begin{equation}\label{eq:333}
  A=\overline{A\cap E}\cap\overline{A\cap H_n}.
  \end{equation}
  Since $f$ is barely continuous, there exists a point $p$ of continuity of the restriction $g=f|_A$. It follows from (\ref{eq:333}) that $A\subseteq \overline{A\cap E}=\overline{g^{-1}(F)}$. The continuity of $g$ at the point $p$ implies that $f(p)=g(p)\in \overline{g(g^{-1}(F))}\subseteq F$. Similarly, $f(p)=g(p)\in \overline{g(g^{-1}(F_n))}\subseteq F_n$, which contradicts to the equality  $F\cap F_n=\emptyset$.

Hence, for every $n\in\omega$ there exists an $H$-set $B_n\subseteq X$ such that
$$
f^{-1}(F)\subseteq B_n\subseteq X\setminus f^{-1}(F_n).
$$
Then
$$
f^{-1}(F)=\bigcap_{n\in\omega} B_n.
$$
Therefore, $f$ is $H_\sigma$-measurable.
\end{proof}

\begin{proposition}\label{th:H1B}
  Let $X$ be a perfect paracompact space, $Y$ be a perfect space and $f:X\to Y$ be an $H_\sigma$-measurable map. Then  $f$ is $F_\sigma$-measurable.
\end{proposition}

\begin{proof} Fix a closed set $F\subseteq Y$. There exists a sequence $(B_n)_{n\in\omega}$ of $H$-subsets of $X$ such that $f^{-1}(F)=\bigcap_{n\in\omega}B_n$. According to \cite{Karlova:FRes:2011} every $B_n$ is $G_\delta$ in  $X$. Therefore,  $f^{-1}(F)$ is $G_\delta$ and $f\in {\rm H}_1(X,Y)$.
\end{proof}

Propositions~\ref{prop:barely_is_Hmeas} and  \ref{th:H1B} imply the following fact.
\begin{corollary}\label{cor:3}
     Let $X$ be a perfect paracompact space, $Y$ be a perfect space. Then every barely continuous map $f:X\to Y$   is $F_\sigma$-measurable.
\end{corollary}

Now we establish the inverse implication in some special cases.

\begin{definition}{\rm
A transfinite sequence  $\mathscr U=(U_\xi:\xi\in[0,\alpha])$ of subsets of a topological space $X$  is {\it regular in $X$}, if
\begin{enumerate}[label=(\alph*)]
  \item each $U_\xi$ is open in $X$;

  \item $\emptyset=U_0\subset U_1\subset U_2\subset\dots\subset U_\alpha=X$;

  \item\label{it:c} $U_\gamma=\bigcup_{\xi<\gamma} U_\xi$ for every limit ordinal $\gamma\in[0,\alpha)$.
\end{enumerate}
For every $\xi\in[0,\alpha)$ we put
$$
P_{\xi+1}=U_{\xi+1}\setminus U_\xi.
$$
Then the partition $\mathscr P=(P_\xi:\xi\in[1,\alpha))$ of $X$ is called {\it regular partition associated with $\mathscr U$.}}
\end{definition}

It is well-known~\cite[12.II]{Kuratowski:Top:1} that a set $A$ is resolvable in the sense of Hausdorff if and only if  there exists an ordinal $\alpha$ and a    decreasing sequence $(F_\xi)_{\xi\in[0,\alpha)}$ of closed subsets of $X$ such that
\begin{gather}
   A=\bigcup_{\xi<\alpha,\,\xi {\footnotesize \mbox{\,\,is odd}}}(F_{\xi}\setminus F_{\xi+1}),\label{g:1}\\
   F_0=X\label{g:2},\\
   F_\gamma=\bigcap_{\xi<\gamma}F_\xi\,\,\,\mbox{if}\,\, \gamma \,\,\mbox{is a limit ordinal}\,\,\mbox{or if}\,\,\gamma=\alpha.\label{g:3}
\end{gather}
Clearly, if $U_\xi=X\setminus F_\xi$ for every $\xi<\alpha$, then the sequence $(U_\xi:\xi<\alpha)$ is regular.

\begin{lemma}\label{lem:dense_base}
  If $\mathscr P=(P_\xi:\xi\in[0,\alpha))$ is a regular partition of a topological space $X$, then the set
  \begin{gather}\label{gath:2}
  U_{\mathscr P}=\bigcup_{\xi\in[0,\alpha)} {\rm int}P_\xi
  \end{gather}
  is dense in $X$.
\end{lemma}

\begin{proof}
  Fix open nonempty set $G\subseteq X$ and let $\beta=\min\{\xi: G\cap P_\xi\ne\emptyset\}$. Since the set $V=\bigcup_{\xi\le\beta}P_\xi$ is open in $X$, the set $V\cap G=P_\beta\cap G$ is open in $X$. Then $G\cap {\rm int}P_\xi\ne\emptyset$. Therefore, $\overline{U_{\mathscr P}}=X$.
\end{proof}

\begin{definition}
  {\rm A family $\mathscr B$ of subsets of a topological space $X$ is said to be {\it a base} for a map $f:X\to Y$, if for every open set $V\subseteq Y$ there exits a subfamily $\mathscr B_V$ of $\mathscr B$ such that $f^{-1}(V)=\bigcup_{B\in\mathscr B_V}B$.}
\end{definition}

\begin{lemma}\label{lem:continuity}
  Let $\mathscr B=\bigcup_{n\in\omega}\mathscr B_n$ be a base for a map $f:X\to Y$ between topological spaces and each family $\mathscr B_n$ consists of mutually disjoint sets. Then
  \begin{gather}\label{gath:3}
    \bigcap_{n\in\omega} U_{\mathscr B_n}\subseteq C(f),
  \end{gather}
  where $C(f)$ means the set of all continuity points of $f$ and the sets $U_{\mathscr B_n}$ are defined by (\ref{gath:2}).
\end{lemma}

\begin{proof}
  Fix $x\in X$ such that $x\in U_{\mathscr B_n}$ for all $n\in\omega$ and let $V$ be an open neighborhood of $f(x)$ in $Y$. Since $\mathscr B$ is a base for $f$, there exists $n\in\omega$ and $B\in\mathscr B_n$ such that $x\in B\subseteq f^{-1}(V)$. Then $x\in B\cap U_{\mathscr B_n}$. Consequently, $x\in {\rm int} B\subseteq f^{-1}(V)$, which implies that $f$ is continuous at the point $x$.
\end{proof}

\begin{definition}
  {\rm  We say that a map $f:X\to Y$ between topological spaces  $X$ and $Y$ is {\it $\sigma$-regular}, if there exists a sequence $(\mathscr P_n)_{n\in\omega}$ of regular partitions of $X$ such that $\bigcup_{n\in\omega}\mathscr P_n$ is a base for $f$.}
\end{definition}

\begin{proposition}\label{prop:sigmareg_is_bc}
  If $X$ is a hereditarily Baire space and $Y$ is a topological space, then every $\sigma$-regular map $f:X\to Y$ is barely continuous.
\end{proposition}

\begin{proof}
Let $F\subseteq X$ be a closed nonempty set. Notice that $F$ is a Baire space. This fact combining with Lemmas~\ref{lem:dense_base} and~\ref{lem:continuity} imply that the set $C(f|_F)$ is dense in $F$, which completes the proof.
\end{proof}

\begin{proposition}\label{prop:hsigma_is_regular}
 Let $X$ be a topological space and $Y$ be a second countable space. Then every $H_\sigma$-measurable map $f:X\to Y$ is $\sigma$-regular.
\end{proposition}

\begin{proof}
 Consider a countable base $(V_n)_{n\in\omega}$ of $Y$. Let $n\in\omega$ be fixed. Since $f^{-1}(V_n)$ is an $H_\sigma$-subset of $X$, the equalities (\ref{g:1})--(\ref{g:3}) imply that there exists a regular partition  $\mathscr P_n$ of $X$ such that $f^{-1}(V_n)=\bigcup_{P\in\mathscr P_n'}P$ for some $\mathscr P'\subseteq \mathscr P$. Hence, $f$ is $\sigma$-regular.
\end{proof}

\begin{theorem}\label{thm:F_sigma_is_bc}
  Let $X$ be a perfectly normal space, $(Y_n:n\in\omega)$ be a covering of $Y$ by closed Polish spaces and $f:X\to Y$. Consider the following conditions:
  \begin{enumerate}
    \item[(a)] $f$ is $F_\sigma$-measurable,

    \item[(b)] $f$ is barely continuous.
  \end{enumerate}
  If $X$ is hereditarily Baire, then (a) $\Rightarrow$ (b). If $X$ is paracompact, then (b) $\Rightarrow$ (a).
\end{theorem}

\begin{proof} (a) $\Rightarrow$ (b). Let $X$ be a hereditarily Baire space. Since $f$ is $F_\sigma$-measurable, every set $X_n=f^{-1}(Y_n)$ is $G_\delta$ in $X$. Then $X_n$ is 1-embedded and well 1-embedded in $X$. Hence, \cite[Theorem 7.3]{Karlova:CMUC:2013} implies that for every $n\in\omega$ there exists an $F_\sigma$-measurable map $g_n:X\to Y_n$ such that $g_n|_{X_n}=f|_{X_n}$.

Clearly, every $g_n:X\to Y_n$ is $H_\sigma$-measurable. Proposition~\ref{prop:hsigma_is_regular} implies that every $g_n$ is a $\sigma$-regular map.

Let $V$ be an open subset of $Y$. The equalities
$$
f^{-1}(V)=\bigcup_{n\in\omega} f^{-1}(V\cap Y_n)=\bigcup_{n\in\omega} g_n^{-1}(V\cap Y_n)
$$
imply that $f:X\to Y$ is $\sigma$-regular.

Now Proposition~\ref{prop:sigmareg_is_bc} implies that $f$ is barely continuous.

(b) $\Rightarrow$ (a). Assume that $X$ is a perfect paracompact space. Since $Y$ is perfect as a union of countably many closed perfect spaces, $f$ is $F_\sigma$-measurable by Corollary~\ref{cor:3}.
\end{proof}

\begin{definition}
 {\rm A topological space $Y$ is called  {\it an adhesive for a topological space $X$} (this fact is denoted as $Y\in {\rm Ad}(X)$), if for any disjoint zero-sets  $A$ and $B$ of $X$ and any continuous maps $f,g:X\to Y$ there exists a continuous map $h:X\to Y$ such that $h|_A=f|_A$ and $h|_B=g|_B$.}
\end{definition}
 It is worth noting that every topological space $Y$ is an adhesive for every strongly zero-dimensional space; a path-connected space  $Y$ is an adhesive for any compact space $X$ in which every point has a base of neighborhoods with discrete boundaries;  $Y$ is an adhesive for an arbitrary topological space  $X$ if and only of $Y$ is contractible (see \cite[Theorem 2.7]{Karlova:Mykhaylyuk:Stable}).

\begin{lemma}\label{lem:adh}
  Let a topological space $Y$ be an adhesive for a topological space $X$,  $\alpha\in [0,\omega_1)$, $(X_n:n\in\omega)$ be a partition of $X$ by functionally ambiguous sets of the $\alpha$'th class and $f_n\in {\rm B}_\alpha(X,Y)$ for all $n\in\omega$.  Then the mapping $f:X\to Y$ defined by
  $$
  f|_{X_n}=f_n|_{X_n}\quad {\rm for}\,\,\,{\rm all}\,\,\,n\in\omega,
  $$
  belongs to  ${\rm B}_\alpha(X,Y)$.
\end{lemma}

\begin{proof}
We will argue by the induction. Let $\alpha=0$. Then every set $X_n$ is clopen in $X$ and every map $f_n$ is continuous on $X$. Clearly, the map $f$ is continuous. Let us observe that we did not use the inclusion $Y\in {\rm Ad}(X)$.

Now let $\alpha =1$. We take a sequence $(f_{n,m})_{m\in\omega}$ of continuous maps $f_{n,m}:X\to Y$ such that $\lim_{m\to\infty} f_{n,m}(x)=f_n(x)$ for all $x\in X$ and $n\in\omega$. For every $n\in\omega$ we choose an increasing sequence $(X_{n,m})_{m\in\omega}$ of functionally closed sets $X_{n,m}$ such that $X_n=\bigcup_{m\in\omega} X_{n,m}$. Since $Y\in {\rm Ad}(X)$, for every $m\in\omega$ there exists a continuous function $f_m:X\to Y$ such that
$$
f_m|_{X_{n,m}}=f_{n,m}|_{X_{n,m}}
$$
for all $n\le m$.  It is easy to see that the sequence $(f_m)_{m\in\omega}$ converges to $f$ on $X$ pointwisely. Thus, $f\in {\rm B}_1(X,Y)$.

Suppose that for some $\alpha\in[2,\omega_1)$ the statement of Lemma is true for all \mbox{$\beta\in[0,\alpha)$}. We take a sequence $(f_{n,m})_{m\in\omega}$ of maps $f_{n,m}\in{\rm B}_{\alpha_n}(X,Y)$ such that $\alpha_n<\alpha$ and $\lim_{m\to\infty} f_{n,m}(x)=f_n(x)$ for all $x\in X$ and $n\in\omega$.
For every $n\in\omega$ we choose an increasing sequence $(X_{n,m})_{m\in\omega}$ of sets of the $\beta_{n,m}$'th functionally multiplicative  class such that $X_n=\bigcup_{m\in\omega} X_{n,m}$ and $\beta_{n,m}<\alpha$. Following the proof of Theorem 2 from \cite[p.~350]{Kuratowski:Top:1} it is not hard to show that  for every $m$ there exists a partition $(A_{n,m}:n\le m)$ of $X$ by functionally ambiguous sets of the classes $<\alpha$ such that $X_{n,m}\subseteq A_{n,m}$.

 By the inductive assumption, for every $m$ there exists a map $f_m\in {\rm B}_{<\alpha}(X,Y)$ such that $f_m|_{A_{n,m}}=f_{n,m}|_{A_{n,m}}$  for all $n\le m$. It is easy to see that the sequence $(f_m)_{m\in\omega}$ converges to $f$ on $X$ pointwisely.  Then $f\in{\rm B}_\alpha(X,Y)$.
\end{proof}

\begin{definition}
  {\rm Let $\alpha\in[0,\omega_1)$. A topological space $Y$ is called {\it (weakly) $B_\alpha$-favorable for $X$}, if ${\rm H}_{\alpha^+}(X,Y)\subseteq {\rm B}_\alpha(X,Y)$ (${\rm K}_{\alpha^+}(X,Y)\subseteq {\rm B}_\alpha(X,Y)$).}
\end{definition}

\begin{theorem}\label{thm:b1=h1}
  Let $X$ be a topological space, \mbox{$\mathscr Y=(Y_n:n\in\omega)$} be a sequentially absorbing covering of a topological space $Y\in{\rm Ad}(X)$ by ${\rm B}_1$-favorable spaces for $X$ and $f:X\to Y$ be a map. If
\begin{enumerate}
  \item[1)] $X$ is a hereditarily Baire  metrizable separable space, $\mathscr Y$ is super sequentially absorbing and $f\in{\rm H}_1(X,Y)$, or

  \item[2)] $X$ is a first countable perfect paracompact space and $f$ is barely continuous,
 \end{enumerate}
 then $f\in{\rm B}_1(X,Y)$.
\end{theorem}

\begin{proof}
   Let us observe that in the second case $f\in{\rm H}_1(X,Y)$ by Corollary~\ref{cor:3}. Applying Proposition~\ref{prop:aux_sigma_metr} we obtain that there exists a covering $(C_n:n\in\omega)$ of $X$ by $F_\sigma$-sets such that $f(C_n)\subseteq Y_n$ for all $n\in\omega$. Then the Reduction Theorem~\cite[p.~350]{Kuratowski:Top:1} implies that there exists a partition $(X_n:n\in\omega)$ of the space $X$ by $F_\sigma$- and $G_\delta$-sets such that $X_n\subseteq C_n$ for all $n\in\omega$. Take any $y_n\in Y_n$ for every $n\in\omega$ and put
   $$
   f_n(x)=\left\{\begin{array}{ll}
                   f(x), & x\in X_n, \\
                   y_n, & x\not\in X_n.
                 \end{array}
   \right.
   $$
   Then $f_n\in{\rm H}_1(X,Y_n)$, since the set $X_n$ is ambiguous. Now $f_n\in{\rm H}_1(X,Y_n)$, because $Y_n$ is ${\rm B}_1$-favorable for $X$. Lemma~\ref{lem:adh} implies that $f\in{\rm B}_1(X,Y)$.
\end{proof}

\section{Extension results}\label{sec:44}

\begin{definition}
  {\rm Let $\mathcal P$ be a property of maps, $\alpha\in[0,\omega_1)$, $X$, $Y$ be topological spaces and $E\subseteq X$. A triple $(X,E,Y)$ has {\it the $\mathcal P$-extension property}, if any map $f:E\to Y$ with the property $\mathcal P$ can be extended to a map $g:X\to Y$ with the property $\mathcal P$.}
\end{definition}

\begin{definition}
  {\rm A subspace $E$ of a topological space $X$ is called {\it $K_\alpha$-embedded in $X$}, if every function $f\in{\rm K}_\alpha(E,\mathbb R)$ can be extended to a function $f\in{\rm K}_\alpha(X,\mathbb R)$.}
\end{definition}

It is known~\cite{KaSp, Karlova:CMUC:2013}  that $E$ is ${\rm K}_\alpha$-embedded in a topological space $X$ in each of the following cases:
\begin{itemize}
\item $E$ is a co-zero subset of $X$,

\item $X$ is completely regular, $E$ is  Lindel\"{o}f and $G_\delta$,

\item $X$ is metrizable, $E$ is $G_\delta$,

\item $X$ is completely regular, $E$ is hereditarily Baire Lindel\"{o}f.
\end{itemize}

\begin{theorem}\label{thm:Ext:Balpha_sigmaMetr}
Let $\alpha\in [1,\omega_1)$, $X$ be a topological space, $E$ be a $K_\alpha$-embedded subspace of  $X$, $(Y_n:n\in\omega)$ be a sequentially absorbing covering of a topological space $Y\in {\rm Ad}(X)$ by Polish weakly $B_\alpha$-favorable spaces for $X$. Then $(X,E,Y)$ has the $B_\alpha$-extension property.
\end{theorem}

\begin{proof} Let $f\in {\rm B}_\alpha(E,Y)$.  According to Proposition~\ref{prop:ext:Balpha} there exists a partition $(C_n:n\in\omega)$ of $E$ by functionally ambiguous sets of the class  $\alpha^+$ in $E$ such that $f|_{C_n}\in{\rm B}_\alpha(C_n,Y_n)$ for all $n\in\omega$. It follows from \cite[Proposition 7.2]{Karlova:CMUC:2013} that there exists a partition $(B_n:\in\omega)$ of $X$ by functionally ambiguous sets of the $\alpha^+$'th class  such that $C_n=B_n\cap E$ for all $n\in\omega$. Now fix $y_n\in Y_n$ and put
\begin{gather*}
  f_n(x)=\left\{\begin{array}{ll}
                  f(x), & x\in C_n, \\
                  y_n, & x\in E\setminus C_n
                \end{array}
  \right.
\end{gather*}
for all $n\in\omega$. Since  $f\in {\rm K}_{\alpha^+}(E,Y)$, $f_n\in {\rm K}_{\alpha^+}(E,Y_n)$. Theorem~7.3 from~\cite{Karlova:CMUC:2013} implies that there exists an extension $g_n\in {\rm K}_{\alpha^+}(X,Y_n)$ of $f_n$ for every $n\in\omega$.  Since $Y_n$ is weakly ${\rm B}_\alpha$-favorable for $X$, $g_n\in {\rm B}_\alpha(X,Y_n)$ for all $n\in\omega$. It follows from Lemma~\ref{lem:adh} that the mapping $g:X\to Y$ defined by
$$g|_{B_n}=g_n|_{B_n}\quad{\rm for}\,\,\,{\rm all}\,\,\,\,n\in\omega$$
belongs to ${\rm B}_\alpha(X,Y)$. Clearly, $g$ is an extension of $f$.
\end{proof}

\begin{lemma}\label{lem:separable_image}
Let  $X$ be a hereditarily Baire metrizable separable space, $Y$ be a metrizable space and  $f:X\to Y$ be a barely continuous map. Then the space $f(X)$ is separable.
\end{lemma}

\begin{proof} We consider an arbitrary $G_\delta$-set $A\subseteq X$  and put $B=\overline{A}$. Since the set  $C(f|_B)$ of all points of continuity of the restriction $f|_B$ is dense $G_\delta$-subset of $B$, then the set $B\setminus C(f|_B)$ is of the first category in $B$. Then the set $D(f|_A)$ of all discontinuities is of the first category in  $A$. Therefore, the set $C(f|_A)$ is dense in $A$, since the space $A$ is Baire. Hence, $f(X)$ is separable by~\cite[p.~398]{Kuratowski:Top:1}.
\end{proof}

\begin{theorem}\label{thm:ext1} Let $X$ be a completely regular space, $E$ be a hereditarily Baire metrizable separable space and $Y$ be a topological space which admits a super sequentially absorbing covering $(Y_n:n\in\omega)$ by completely metrizable spaces $Y_n$. Then  $(X,E,Y)$ has the ${\rm K}_1$-extension property.
\end{theorem}

\begin{proof} Let $f\in {\rm K}_1(E,Y)$. By Proposition~\ref{prop:aux_sigma_metr} there exists a sequence   $(A_n)_{n\in\omega}$ of  $F_\sigma$-subsets of $E$ such that $E=\bigcup_{n\in\omega} A_n$ and $f(A_n)\subseteq Y_n$ for all $n\in\omega$.

Notice that $f$ is barely continuous accordingly to \cite[Theorem 4.12]{Kum}. Then  $f(A_n)$ is separable by Lemma~\ref{lem:separable_image}. Therefore, the space  $Z_n=\overline{f(A_n)}$ is Polish. Then Theorem 7.3. from~\cite{Karlova:CMUC:2013} implies that $(X,E,Z_n)$ has the ${\rm K}_1$-extension property for every $n\in\omega$. Completely similar as in the proof of Theorem~\ref{thm:Ext:Balpha_sigmaMetr} one can construct a map $g\in {\rm K}_1(X,Y)$ which is an extension of  $f$.
\end{proof}

The following question is open. 

\begin{question}
Let $X$ be a metrizable space, \mbox{$\mathscr Y=(Y_n:n\in\omega)$} be a super sequentially absorbing covering of a topological space $Y$, $f:X\to Y$ be an $F_\sigma$-measurable map and $X$ be separable or hereditarily Baire. Is it true that there exists a covering $(C_n:n\in\omega)$ of $X$ by $F_\sigma$-sets $C_n$ such that $f(C_n)\subseteq Y_n$ for all $n\in\omega$?
\end{question}

\end{document}